\documentclass[a4paper]{article}

\usepackage{hyperref}

\usepackage{amsmath,amssymb,amsthm}

\numberwithin{equation}{section}

\usepackage{shortvrb, makeidx}
\newtheorem{definition}{\hspace{2em}Definition}[section]

\newtheorem{theorem}[definition]{\hspace{2em}Theorem}

\newtheorem{lemma}[definition]{\hspace{2em}Lemma}

\newtheorem{remark}{\hspace{2em}Remark}[section]

\makeindex

\title{Well-posedness of 1-D compressible Euler-Poisson equations with physical vacuum}
\author{Xumin Gu\footnote{School of Mathematical Sciences, Fudan University, Shanghai 200433, P.R.China. {\it Email:082018045@fudan.edu.cn}}\ \  and Zhen Lei\footnote{School of Mathematical Sciences; LMNS and Shanghai Key Laboratory for Contemporary Applied Mathematics, Fudan University, Shanghai 200433, P. R.China. {\it Email: leizhn@yahoo.com, zlei@fudan.edu.cn}}}

\date{}

\begin{document}
\maketitle

\begin{abstract}
This paper is concerned with the 1-D compressible Euler-Poisson equations with moving physical vacuum boundary condition. It is usually used to describe the motion of a self-gravitating inviscid gaseous star. The local well-posedness of classical solutions is established in the case of the adiabatic index $1<\gamma<3$.
\end{abstract}

\section{Introduction}
The motion of self-gravitating inviscid gaseous stars in the universe can be described by following free boundary problem for the compressible Euler equations coupled with Poisson equation:
\begin{align}
  \label{euler-poisson-3d}
  \rho_t + \nabla \cdot (\rho u) &= 0&\text{in}\ \ \Omega(t),\\
  \rho [u_t + u \cdot \nabla u] + \nabla P &= \rho \nabla \phi &\text{in}\ \  \Omega(t),\\
  - \Delta\phi &= 4 \pi  \rho g&\text{on}\ \  \Omega(t),\\
  \nu (\Gamma(t)) &= u,\\
  (\rho, u) &= (\rho_0, u_0) &\text{on}\ \  \Omega(0).\label{5353}
\end{align}
The open, bounded domain $\Omega(t) \subset \mathbb{R}^3$ denotes the changing domain occupied by the gas. $\Gamma(t) := \partial \Omega(t)$ denotes the moving vacuum boundary, $\nu (\Gamma(t))$ denotes the velocity of $\Gamma(t)$. The density $\rho > 0$ in $\Omega(t)$ and $\rho = 0$ in $\mathbb{R}^3\setminus \Omega(t)$. $u$ denotes the Eulerian velocity field. $p$ denotes the pressure function, and $\rho$ denotes the density of the gas. $\phi$ is the potential function of the self-gravitational force, and g is the gravitational constant. The equation of state for a polytropic gas is given by:
\begin{equation}
P= C_{\gamma}\rho^{\gamma} \ \ \ \text{for}\ \ \gamma > 1, 
\label{poly}
\end{equation}
where $C_{\gamma}$ is the adiabatic constant which we set to be one. For more details of the related background, see for instance, \cite{C}.

With the sound speed given by $c^2:=\sqrt{\partial P / \partial \rho}$, and with $c_0=c(\cdot,0)$, the condition
\begin{equation}
 - \infty < \dfrac{\partial c_0^2}{\partial N} <0 \ \ \text{on}\ \ \  \Gamma
 \label{vacuum}
\end{equation}
defines a ``physical vacuum'' boundary, where $N$ denoting the outward unit normal to the initial boundary $\Gamma := \partial \Omega(0)$. This definition of physical vacuum was motivated by the case
of Euler equations with damping studied in \cite{L_1996}\cite{LT_2000} and the physical
vacuum behavior can be realized by some self-similar solutions and stationary
solutions for different physical systems such as Euler equations with damping. For more
details and the physical background of this concept, please see \cite{L_1996}\cite{LT_1997}\cite{TY_2006}.

The local existence theory of classical solutions featuring the physical vacuum boundary even for one-dimensional compressible Euler equations was only established recently. This is because if the physical vacuum boundary condition is assumed, Euler equations becomes a degenerate and characteristic hyperbolic system and the classical theory of hyperbolic systems can not be
directly applied. In \cite{MJ_2009},  Jang and Masmoudi consider the one-dimensional Euler equations in mass Lagrangian coordinates. Local existence was proved using a new structure lying upon the physical vacuum in the framework of free boundary problems. Independently of this work, in \cite{DS_2009}, Coutand and Shkoller constructed $H^2$-type solutions with moving boundary in Lagrangian coordinates based on Hardy inequalities and degenerate parabolic regularization.

In this paper, we will focus on the 1-D case for the system (\ref{euler-poisson-3d})--(\ref{5353}) with the physical vacuum condition:
\begin{align}
  \label{euler-poisson}
  \rho_t + (\rho u)_{\eta} &= 0 &\text{in}&\ \ I(t),\\
  \rho [u_t + u u_{\eta}] + (\rho^{\gamma})_{\eta} &= \rho \phi_{\eta}&\text{in}&\ \  I(t),\label{extra}\\
  - \phi_{\eta\eta} &= C \rho &\text{on}&\ \  I(t),\label{9090}\\
   (\rho, u) &= (\rho_0, u_0) &\text{on}&\ \  I(0),\\
  \nu (\Gamma(t)) &= u,\label{98}\\
0 < \bigg|\dfrac{\partial c_0^2}{\partial \eta}&\bigg| <+\infty \ \ &\text{on}&\ \ \  \Gamma,\label{1dvacuum}
\end{align}
where $I(0) = I = \{0 < \eta <1\}$ and $\Gamma:=\partial I$ and prove the local existence result for it (The 3D case will appear soon).

Our main result is the following theorem:
\begin{theorem}\label{theorem}(Local wellposedness) For $1<\gamma<3$, assume that initial data  $\rho_0 >0$ in $I$, $M_0 < \infty$ (defined in (\ref{M0000})), and the physical vacuum condition (\ref{1dvacuum}) holds. Then there exists a unique solution to (\ref{euler-poisson-lag})--(\ref{5r}) (and hence (\ref{euler-poisson})--(\ref{98})) on $[0,T]$ for some sufficiently small $T >0$  such that 
\begin{equation}
\sup_{t \in [0,T]} E(t) \leq 2M_0.
\end{equation}
\end{theorem}

The local wellposedness result for the corresponding Euler equation was proved by Coutand and Shkoller in \cite{DS_2009}. Motivated by \cite{DS_2009}, we also use the Lagrangian coordinates to reduce the original system to that in a fixed domain. In our problem, we have the extra potential force term $\phi_{\eta}$ in (\ref{extra}). To handle this term, we will give an explicit formula for it and show that it is a function of $\rho_0$. Then we construct the approximate solution in two steps. Firstly, using Galerkin scheme to find the solution to linearized problem of the degenerate parabolic regularization. In this process, we would make fundamental use of the higher-order Hardy-type inequality introduced in \cite{DS_2009} (we would give a new proof). But we will define an intermediate variable which is different to the one used in \cite{DS_2009}. By using our intermediate variable, the improvement of the space regularity for the solution of linear problem will be easy and clear with less computation. Secondly, using fixed-point scheme to get the solution to the degenerate parabolic regularization. Last we would derive a priori estimates for the approximate solution. This part is more or less similar to that in \cite{DS_2009}. For a self-contained presentation, we will still carry out the proof in Section \ref{ee}.
Below we will mainly focus on the case of $\gamma=2$. The general case for $1<\gamma<3$ is treated in Section \ref{ga3}.

Now we briefly review some related theories and results from various aspects. For Euler-Poisson equations, the existence theory for the stationary solutions has been proved by Deng, Liu, Yang, and Yao in \cite{TY_2002}. For Navier-Stokes-Poisson equations, Li, Matsumura and Zhang\cite{HL_2010} studied optimal decay rate for the system and Zhang, Fang studied global behavior for spherically symmetric case with degenerate viscosity coefficients in \cite{TD}. 

For compressible fluids, Makino proved the local-in-time existence of solution in \cite{Makino_1986} with boundary condition $\rho=0$ for some non-physical restrictions on the initial data. And Lindblad proved the local-in-time existence with vacuum boundary condition $P=0$ for general case of initial data with the main tool which is the passage to the Lagrangian coordinates for reducing the original problem to that in a fixed domain in \cite{Lind_2005}. And H.L. Li, J. Li, Xin \cite{HL_2008}, Luo, Xin, Yang \cite{XinL}, Xin\cite{Xin} also did many works on compressible Navier-Stokes equation with vacuum.

For incompressible flows, Wu solved local well-posedness for the irrotational problem, with no surface tension in all dimensions in \cite{Wu_1997} and \cite{Wu_1999}. Lindblad proved local existence of solutions general problem with no surface tension, assuming the Rayleigh-Taylor sign condition for rotational flows in \cite{Lindblad_2005}. For the problem with surface tension, B.Schweizer proved existence for the general three-dimensional irrotational problem in \cite{Schweizer_2005}. And we also mention the works by Ambrose and Masmoudi \cite{AM}, Coutand and Shkoller \cite{DS}, and P. Zhang and Z.Zhang \cite{ZZ}.

This paper is organized as follows: In Section \ref{s2}, we formulate the problem in Lagrangian coordinates. In Section \ref{s3}, we present some lemmas will be used. In Section \ref{s4}--Section \ref{s5}, we introduce a degenerate parabolic approximation and solve it by a fixed-point method. In Section \ref{ee} --\ref{s7}, we derive the a priori estimates and prove the local well-posedness for $\gamma=2$. In Section \ref{ga3}, we discuss the general case for $1 < \gamma <3$.

\section{Lagrangian Formulation}\label{s2}

Here, we denote $\eta$ as Eulerian coordinates and denote $x$ as Lagrange coordinates, which means $\eta(x,t)$ denotes the ``position'' of the gas particle $x$ at time $t$: 
\begin{equation}
  \partial_t \eta = u \circ \eta \ \ \text{for}\ \  t > 0 \ \ \text{and}\ \  \eta(x,0)=x,
  \label{hhhhh}
\end{equation}  
where $\circ$ denotes the composition $[u\circ\eta](x,t)=u(\eta(x,t),t)$. We also have: 
\begin{equation}
  \begin{split}
    v &= u\circ \eta \ \ (\text{Lagrangian velocity}),\\
    f &= \rho \circ \eta \ \ (\text{Lagrangian density}),\\
    \Phi &= \phi \circ \eta \ \ (\text{Lagrangian potential field}).
  \end{split}
  \label{lagrang}
\end{equation}

\subsection{Fixing the domain and the Lagrangian version of the system}
Noticing (\ref{hhhhh}) and (\ref{lagrang}), the Lagrangian version of system (\ref{euler-poisson})--(\ref{98}) can be written  on the fixed reference domain $I$ as
\begin{align}
  \label{euler-poisson-lag}
  f_t + f \dfrac{\partial}{\eta_x\partial_x} v &= 0 &\text{in}&\ \ I\times (0,T],\\
  f v_t + \dfrac{\partial}{\eta_x\partial_x}(f^{2}) &= f \dfrac{\partial}{\eta_x\partial_x}\Phi &\text{in}&\ \  I\times (0,T],\\ 
  - (\dfrac{\partial}{\eta_x\partial_x})^2\Phi &= C f &\text{in}&\ \  I\times (0,T],\\
  (f, v, \eta) &= (\rho_0, u_0, e) &\text{in}&\ \ I\times \{t=0\},\label{5r}
\end{align}
where $e(x)=x$ denotes the identity map on $I$.

By conservation law of mass, we have
\begin{equation}
  f = \rho \circ \eta = \rho_0 / \eta_x.
  \label{lag_mass}
\end{equation}
Hence, the initial density function $\rho_0$ can be viewed as a parameter in the Euler equations.

Since $\rho_0 >0$ in $I$, (\ref{1dvacuum}) implies that for some positive constant $C$ and $x \in I$ near the vacuum boundary $\Gamma$,
\begin{equation}
  \rho_0 \geq C \text{dist}(x, \Gamma).
  \label{cond_1}
\end{equation}

Hence, for every $x \in I$, we have:
\begin{align}
  \bigg |\dfrac{\partial\rho_0}{\partial x}(x)\bigg |\geq C\ \ \text{when}  \ d(x,\partial I) \leq \alpha,
  \label{cond_2}\\
  \rho_0 \geq C_{\alpha} > 0\ \ \text{when}  \ d(x,\partial I) \geq \alpha
  \label{cond_3}
\end{align}
for some $\alpha >0$ and a constant $C_{\alpha}$ depending on $\alpha$.

In summary, we write the compressible Euler-Poisson System as 
\begin{align}
  \label{lag}
  \rho_0 v_t + (\rho_0^{2}/\eta_x^{2})_x  &= \rho_0 \Phi_x / \eta_x &\text{in}&\ \  I\times (0,T],\\ \label{ffffff}
  - (\Phi_x/\eta_x)_x  &= C \rho_0 &\text{in}&\ \  I\times (0,T], \\ \label{field}
(v, \eta) &= (u_0, e) &\text{in}&\ \ I\times \{t=0\},\\
  \rho_0 &= 0 &\text{on}&\ \  \Gamma,
\end{align}
with $\rho_0 \geq C \text{dist}(x,\Gamma)$ for $x \in I $ near $\Gamma$.

\subsection{The formula for potential force}
Now we try to give an explicit formula for the potential force $\phi_{\eta}$ in (\ref{euler-poisson}) and corresponding term $\Phi_x /\eta_x$ in (\ref{lag}). Set $I(t) = (a(t),b(t))=(\eta(0,t),\eta(1,t))$.
 
First, for every $t$, we can directly get 
\begin{equation}
  \phi_{\eta}(\eta,t) = - \int_{a(t)}^\eta \rho(y,t)dy + M(t).
  \label{}
\end{equation}
It is reasonable to assume that $|\phi_{\eta}(-\infty)|=|\phi_{\eta}(+\infty)|$. Since the gas only occupied bounded interval, the force of gas produced in $-\infty$ and $+\infty$ can be regarded as the same large (see \cite{Sh_2001}). Noticing $\phi_{\eta\eta} = -C\rho \leq 0$, so we get:
\begin{align}
\phi_{\eta}(+\infty)&=-\dfrac{1}{2}\int_{-\infty}^{+\infty} \rho(y,t)dy,\\
\phi_{\eta}(-\infty)&=\dfrac{1}{2}\int_{-\infty}^{+\infty} \rho(y,t)dy,\\
M(t) &= \dfrac{1}{2}\bigg(\int_{a(t)}^{+\infty} \rho(y,t)dy-\int_{-\infty}^{a(t)} \rho(y,t)dy\bigg).
\end{align}
Since $\rho(\eta,t)>0$ in $I(t)$ and $\rho(\eta,t)=0$ when $\eta \leq a(t) $ or $\eta \geq b(t)$, we have
\begin{equation}
  \phi_{\eta}(\eta,t) = - \int_{a(t)}^{\eta} \rho(y,t)dy + \dfrac{1}{2}\int_{a(t)}^{b(t)}\rho(y,t)dy.
  \label{force}
\end{equation}
Then we transform the formula (\ref{force}) to Lagrange variables:
\begin{equation}
  \begin{split}
  \Phi_x /\eta_x(x,t)&= \phi_{\eta}(\eta(x,t),t) \\&=  - \int_{a(t)}^{\eta(x,t)} \rho(y,t)dy + \dfrac{1}{2}\int_{a(t)}^{b(t)} \rho(y,t) dy\\
&=- \int_{\eta(0,t)}^{\eta(x,t)} \rho(\eta(z,t),t)d\eta(z,t) + \dfrac{1}{2}\int_{\eta(0,t)}^{\eta(1,t)} \rho(\eta(z,t),t) d\eta(z,t)\\
  &=- \int_{0}^{x} f(z,t) \eta_z dz + \dfrac{1}{2}\int_{0}^{1} f(z,t) \eta_z dz.
  \end{split}
  \label{lag_force}
\end{equation}
With (\ref{lag_mass}), we can finally get
\begin{equation}
  F:=\dfrac{\Phi_x}{\eta_x} = - \int_{0}^{x} \rho_0(y)dy + \dfrac{1}{2}\int_{0}^{1} \rho_0(y)dy.
  \label{forceF}
\end{equation}

\begin{remark}
If $\rho_0 \in C^{\alpha}$, then $F \in C^{1+\alpha}$, we will see that this regularity is important for the case $\gamma \neq 2$ in Section \ref{ga3}.
\end{remark}

With formula (\ref{forceF}), we can write the whole system as 

\begin{align}
  \label{simsystem}
  \rho_0 v_t + (\rho_0^{2}/\eta_x^{2})_x  &= \rho_0 F &\text{in}&\ \ I \times (0,T],\\ 
  (v,\eta) &=(u_0,e) &\text{in}&\ \ I \times \{t=0\},\\
 \rho_0 &= 0 &\text{on}&\ \  \Gamma, \label{90909}
\end{align}
with $\rho_0 \geq C \text{dist}(x,\Gamma)$ for $x \in I $ near $\Gamma$.

\subsection{The higher-order energy function} The higher-order energy function is defined as follows:
\begin{equation}
  \begin{split}
    E(t;v)&=\sum_{s=0}^4 ||\partial_t^s v(t,.)||_{H^{2-\frac{s}{2}}(I)}^2 + \sum_{s=0}^2||\rho_0\partial_t^{2s}v(t,.)||_{H^{3-s}(I)}^2 \\
  &+||\sqrt{\rho_0}\partial_t\partial_x^2 v(t,.)||_{L^2(I)}^2+||\rho_0^{\frac{3}{2}}\partial_t\partial_x^3 v(t,.)||_{L^2(I)}^2\\&+||\sqrt{\rho_0}\partial_t^3\partial_x v(t,.)||_{L^2(I)}^2+||\rho_0^{\frac{3}{2}}\partial_t^3\partial_x^2 v(t,.)||_{L^2(I)}^2.
\end{split}
  \label{energy}
\end{equation}
Let $P$ denotes a generic polynomial function of its arguments whose meaning may change from line to line. Let 
\begin{equation}
M_0=P(E(0;v)).
\label{M0000}
\end{equation}

\section{Weighted Spaces and A higher-order Hardy-type inequality}\label{s3}
\subsection{Embedding of a Weighted Sobolev Space}
Using $d$ to denote the distance function to the boundary $\Gamma$, and letting $p=1$ or $2$, the weighted Sobolev space $H_{d^p}^1(I)$, with norm given by $\big(\int_Id(x)^p(|R(x)|^2+|\partial_xR(x)|^2)dx\big)^{\frac{1}{2}}$ for any $R \in H_{d^p}^1(I)$, satisfies the following embedding:
\begin{equation}
  H_{d^p}^1(I) \hookrightarrow H^{1-\frac{p}{2}}(I).
  \label{}
\end{equation}
So that there is a constant $C>0$ depending only on $I$ and $p$ such that
\begin{equation}
  ||R||_{1-p/2}^2\leq C\int_Id(x)^p(|R(x)|^2+|\partial_xR(x)|^2)dx.
  \label{emb}
\end{equation}
See, for example, Section 8.8 in \cite{Kufner_1985}.
\subsection{A higher-order Hardy-type inequality}
The following two lemmas can be found in \cite{DS_2009}. We will use Lemma \ref{hardy} to construct the approximate solution in Section \ref{s5} and use Lemma \ref{para_lema} to obtain estimates independent of the regularization parameter defined in Section \ref{s4}.
\begin{lemma}
  Let $s\geq 1$ be a given integer, and suppose that
  \begin{equation}
    u \in H^s(I) \cap H_0^1(I),
    \label{}
  \end{equation}
and d is the the distance function to $\partial I$, we have that $\frac{u}{d} \in H^{s-1}(I)$ with
  \begin{equation}
    \bigg|\bigg|\frac{u}{d}\bigg|\bigg|_{H^{s-1}} \leq C ||u||_{H^s}.
    \label{}
  \end{equation}
  \label{hardy}

\end{lemma}
\begin{proof}
Let $u \in H^s(I)\cap H_0^1(I)$, then for $0\leq m\leq s$:
\begin{equation}
\partial_x^m\bigg(\dfrac{u}{d}\bigg)=\dfrac{f}{d^{m+1}},
\end{equation}
where 
\begin{equation}
f=\sum_{k=0}^m C_m^k \partial_x^{m-k}u (-1)^{k}k!d^{m-k}.
\end{equation}

With simple calculation, we can get
\begin{equation}
\partial_x f = \partial_x^{m+1}u(-1)^m m!d^{m}.
\end{equation}

Now using the fundamental calculus theorem, when $0 \leq x\leq \frac{1}{2}$, we have:
\begin{align}
f(x)&=f(0)+\int_0^x \partial_x f(y) dy\\&=x\int_0^1\partial_x f(\theta x) d\theta\\
&=(-1)^m m!x^{m+1}\int_0^1\partial_x^{m+1}u(\theta x)\theta^{m}d\theta.
\end{align}
Similarly, when $\frac{1}{2} \leq x \leq 1$, we have:
\begin{align}
f(x)=(-1)^{m+1} m!(1-x)^{m+1}\int_0^1\partial_x^{m+1} u(1-\theta(1-x)) \theta^m d\theta.
\end{align}

Then:
\begin{equation}
\begin{split}
||\partial_x^m\bigg(\dfrac{u}{d}\bigg)||_{L^2}^2&=\int_0^{\frac{1}{2}}(\dfrac{f}{x^{m+1}})^2dx+\int_{\frac{1}{2}}^1(\dfrac{f}{(1-x)^{m+1}})^2dx\\
&=\int_0^{\frac{1}{2}}[(-1)^m m!\int_0^1\partial_x^{m+1}u(\theta x)\theta^{m}d\theta]^2 dx\\&\ \ +\int_{\frac{1}{2}}^1[(-1)^{m+1} m!\int_0^1\partial_x^{m+1} u(1-\theta(1-x)) \theta^m d\theta]^2 dx\\
&\leq C||\partial_x^{m+1}u||_{L^2(I)}^2.
\end{split}
\end{equation}

In this way, we finally get:
\begin{equation}
||\dfrac{u}{d}||_{H^{s}} \leq C ||u||_{H^{s+1}}.
\end{equation}

\end{proof}
\begin{lemma}
  Let $\kappa >0$ and $g \in L^{\infty}(0,T;H^s(I))$ be given, and let $f \in H^1(0,T;H^s(I))$ be such that
  \begin{equation}
    f +\kappa f_t =g \ \ \text{in} \ \ (0,T)\times I.
    \label{}
  \end{equation}
  Then,
  \begin{equation}
    ||f||_{L^{\infty}(0,T;H^s(I))} \leq C \max\{||f(0)||_{H^s(I)},||g||_{L^{\infty}(0,T;H^s(I)}\},
  \end{equation}
  \label{para_lema}
where $C$ is independent of $\kappa$.
\end{lemma}
\section{The degenerate parabolic approximation of the System}\label{s4}

\subsection{Smoothing the initial data}
\label{sm}
For the purpose of constructing solutions, we will smooth the initial velocity
field $u_0$ and density field $\rho_0$ while preserving the conditions $\rho_0 > 0 $ in $I$ and (\ref{cond_1}) firstly.

For $\kappa>0$, let $0 \leq \alpha_{\kappa}(x) \in C_c^{\infty}(\mathbb{R})$ denote the standard family of mollifiers with $spt(\alpha_{\kappa}) =  \{x\big| |x| \leq \kappa\}$ and let $E_I$  denote a Sobolev extension operator mapping $H^s(I)$ to $H^s(\mathbb{R})$ for $s \geq 0$.

Now we set smoothed initial velocity filed $u_0^{\kappa}$ as:
\begin{equation}
u_0^{\kappa} = \alpha_{1/|\ln\kappa|} \ast E_I(u_0),
\label{smu}
\end{equation}
and smoothed initial density function $\rho_0^{\kappa}$ is defined as the solution of the elliptic equation:
\begin{align}
\label{smp}
\partial_x^2 \rho_0^{\kappa} &= \partial_x^2[\alpha_{1/|\ln\kappa|}\ast E_I(\rho_0)] &\text{in}&\ \  I,\\
\rho_0^{\kappa} &= 0 &\text{on}&\ \  \Gamma.
\end{align}

So for sufficiently small $\kappa > 0$, $u_0^{\kappa},\rho_0^{\kappa} \in C^{\infty}(\overline{I})$,       $\rho_0^{\kappa} >0$ in $I$, and vacuum condition (\ref{cond_1}) is preserved. Details can be found in \cite{DS_2009}.

From now on, we will denote $u_0^{\kappa}$ by $u_0$ and $\rho_0^{\kappa}$ by $\rho_0$ for convenience and it is easy to show that Theorem \ref{theorem} holds with the optimal regularity by a standard density argument.

\subsection{Degenerate parabolic approximation}
For notational convenience, we will write 
\begin{equation}
  \eta' = \dfrac{\partial \eta}{\partial x}  
  \label{}
\end{equation}
and similarly for other functions.
Now for $\kappa >0$, we consider the following nonlinear degenerate parabolic approximation of the compressible Euler-Poisson System (\ref{simsystem})--(\ref{90909}):
\begin{align}
\label{tytyt}
  \rho_0 v_t + (\dfrac{\rho_0^2}{\eta'^2})' &= \rho_0F+\kappa(\rho_0^2v')' &\text{in}&\ \ I \times[0,T],\\
  (v,\eta)&=(u_0,e)&\text{in}&\ \ I \times\{t=0\},\\
  \rho_0 &= 0 &\text{on}&\ \ \Gamma
  \label{appro}
\end{align}
with $\rho_0(x) \geq C \text{dist}(x,\Gamma)$ for $x \in I$ near $\Gamma$. We emphasis that the data $(\rho_0,u_0)$ has been smoothed as in Section (\ref{sm}).

We will first obtain the existence of the solution to (\ref{tytyt})--(\ref{appro}) on a short time interval $[0,T_{\kappa}]$ (with $T_{\kappa}$ possibly depending on $\kappa$). Then we will show that the time of existence does not depend on $\kappa$ via a priori estimates in Section \ref{ee} for this sequence of solutions independent of $\kappa$. Then the existence of a solution to the compressible Euler-Poisson system is obtained as the weak limit as $\kappa \to 0$ of the sequence of solutions to (\ref{tytyt})--(\ref{appro}).

\section{Solving the parabolic $\kappa$ - problem by a fixed-point method}\label{s5}

\subsection{Assumption on initial data}
Using the fact that $\eta(x,0)=x$ and $F = - \int_{0}^{x} \rho_0(y)dy + \dfrac{1}{2}\int_{0}^{1} \rho_0(y)dy$, the quantity $v_t|_{t=0}$ for the degenerate parabolic $\kappa$-problem can be computed using (\ref{tytyt}):
\begin{equation}
  \begin{split}
u_1:&=  v_t\bigg|_{t=0}\\&=\bigg(- \int_{0}^{x} \rho_0(y)dy + \dfrac{1}{2}\int_{-\infty}^{+\infty} \rho_0(y)dy+\dfrac{\kappa}{\rho_0}[\rho_0^2v']'-\dfrac{1}{\rho_0}(\dfrac{\rho_0^2}{\eta'^2})'\bigg)\bigg|_{t=0}\\
  &=\bigg(- \int_{0}^{x} \rho_0(y)dy + \dfrac{1}{2}\int_{-\infty}^{+\infty} \rho_0(y)dy+\dfrac{\kappa}{\rho_0}[\rho_0^2u_0']'-2\rho_0'\bigg).
\end{split}
  \label{initialdata}
\end{equation}
Inductively, for all $k \geq 2,\ k \in \mathbb{N}$:
\begin{equation}
\begin{split}
 u_k:&= \partial_t^k v\bigg|_{t=0}=\partial_t^{k-1}\bigg(\dfrac{\kappa}{\rho_0}[\rho_0^2v']'-\dfrac{1}{\rho_0}(\dfrac{\rho_0^2}{\eta'^2})'\bigg)\bigg|_{t=0}.
\end{split}
  \label{vt}
\end{equation}
These formulae make it clear that each $\partial_t^k v|_{t=0}$ is a function of space-derivates of $u_0$ and $\rho_0$.

\subsection{Functional framework for the fixed-point scheme}
For $T>0$, we shall denote by $\mathcal{X}_{T}$ the following Hilbert space:
\begin{equation}
  \begin{split}
  \mathcal{X}_{T} = \{v| v \in W^{5,2}(0,T;H^1(I)) \cap W^{4,2}(0,T;H^2(I)); \\
  \rho_0 v \in W^{5,2}(0,T;H^2(I)) \cap W^{4,2}(0,T;H^3(I))\}.
\end{split}
  \label{space}
\end{equation}
which is endowed with its natural Hilbert norm:
\begin{equation}
  \begin{split}
  ||v||_{\mathcal{X}_{T}}^2 = &||v||_{W^{5,2}(0,T;H^1(I))}^2+||v||_{W^{4,2}(0,T;H^2(I))}^2\\&+||\rho_0v||_{W^{5,2}(0,T;H^2(I))}^2+||\rho_0v||_{W^{4,2}(0,T;H^3(I))}^2.
\end{split}
  \label{norm}
\end{equation}

For $M>0$ given sufficiently large, we can define the following closed, bounded, convex subset of $\mathcal{X}_{T}$:
\begin{equation}
  \mathcal{C}_{T}(M)= \{ v \in \mathcal{X}_{T} : \partial_t^a v|_{t=0} = u_a,a=0,1,2,3,4,5,6,||v||_{\mathcal{X}_{T}}^2 \leq M \}.  
  \label{set}
\end{equation}
which is indeed non-empty if $M$ is large enough which would be determined by initial data. Henceforth, we assume that $T > 0$ is given independently of the choice of $v \in \mathcal{C}_T(M)$, such that
\begin{equation}
  \eta(x,t)= x +\int_0^t v(x,s)ds
  \label{}
\end{equation}
is injective for $t \in [0,T]$, and that $\frac{1}{2} \leq \eta'(x,t) \leq \frac{3}{2}$ for $t \in [0,T]$ and $x \in \overline{I}$. This can be achieved by taking $T >0$ sufficiently small: with $e(x)=x$, notice that
\begin{equation}
  ||\eta'-e||_{H^1} = ||\int_0^t v'(x,s)ds||_{H^1} \leq \sqrt{T}M.
  \label{}
\end{equation}
We will apply the fixed-point methodology in $\mathcal{X}_{T}$ to prove the existence of a solution to the $\kappa$-regularized parabolic problem (\ref{appro}).

Finally, we define a polynomial function  $\mathcal{N}_0$ of norms of the non-smoothed initial data $u_0$ and $\rho_0$ as follows:
\begin{equation}
  \mathcal{N}_0 = P_{\kappa}(||\rho_0||_{L^2},||u_0||_{L^2}),
  \label{N0}
\end{equation}
where $P_{\kappa}$ is a generic polynomial with coefficients dependent on powers of $|\ln\kappa|$.

Using the properties of the convolution (\ref{smu}) and (\ref{smp}), $\forall s \geq 1$, $\forall k \in {1,2,3,4,5,6}$, the quantities defined in (\ref{vt}) (using the smoothed initial data $u_0^{\kappa}$ and $\rho_0^{\kappa}$) satisfies:
\begin{equation}
||u_k||_{H^s} \leq P(||\rho_0^{\kappa}||_{H^{s+k}},||u_0^{\kappa}||_{H^{s+2k}}) \leq C_s P_{\kappa}(||\rho_0||_{L^2},||u_0||_{L^2}) \leq \mathcal{N}_0.
\end{equation}
%

\subsection{Linearizing the degenerate parabolic $\kappa$-problem}
For every $\overline{v} \in \mathcal{C}_{T}(M)$, we define $\overline{\eta}= x +\int_0^t \overline{v}(x,\tau)d\tau$ and consider the linear equation for $v$:
\begin{equation}
  \begin{split}
  \rho_0 v_t - \kappa[\rho_0^2v']'&= -[\dfrac{\rho_0^2}{\overline{\eta}'^2}]'+\rho_0 F,\\
	v(\cdot,0)&= u_0,
  \label{linear}
  \end{split}
\end{equation}
where $F$ is defined in (\ref{forceF}).

In order to use high-order Hardy-type inequality, it will be convenient to introduce the new variable $X=\rho_0 v$, which belongs to $H_0^1(I)$ (can be seen below). Here we choose a different variable $X$ from that used by Coutand and Shkoller in \cite{DS_2009}, which would simplify the process of improving the space regularity for solution of (\ref{linear}).

By a simple computation, we can see that (\ref{linear}) is equivalent to
\begin{equation}
  v_t - \kappa\dfrac{1}{\rho_0}[\rho_0^2v']'= -\dfrac{1}{\rho_0}[\dfrac{\rho_0^2}{\overline{\eta}'^2}]'+ F,  
  \label{}
\end{equation}
and hence that
\begin{align}
\label{stewewe}
  \dfrac{X_t}{\rho_0} - \kappa X'' + \kappa\dfrac{\rho_0''}{\rho_0}X &= G &\text{in}&\ \  I\times [0,T],\\
  X&=0&\text{on}&\ \  \Gamma\times [0,T],\\
  X|_{t=0}&=\rho_0 u_0&\text{in}&\ \  I,
  \label{Xequ}
\end{align}
where
\begin{equation}
  G=F+\dfrac{2}{\overline{\eta}'}(\dfrac{\rho_0}{\overline{\eta}'})'= - \int_{0}^{x} \rho_0(y)dy + \dfrac{1}{2}\int_{0}^{1} \rho_0(y)dy+\dfrac{2}{\overline{\eta}'}(\dfrac{\rho_0}{\overline{\eta}'})'.
  \label{Gdef}
\end{equation}

We shall therefore solve the degenerate linear parabolic problem (\ref{stewewe})--(\ref{Xequ}) with Dirichlet boundary conditions, which (as we will prove) will surprisingly admit a solution with arbitrarily high space regularity (depending on the regularity of $G$ on the right-hand side of (\ref{stewewe}) and the initial data of course), and not just an $H_0^1(T)$-type weak solution. After we obtain the solution $X$, we will then easily find our solution $v$ to (\ref{linear}).

In order to apply the fixed-point theorem, we shall obtain estimates for $v$ with certain high space-time regularity. Here, we study the sixth time-differentiated problem and define the new variable 
\begin{equation}
  Y = \partial_t^6X = \rho_0\partial_t^6 v.
  \label{Ydefi}
\end{equation} 

We consider the following equation for $Y$
\begin{align}
\label{fasss}
  \dfrac{Y_t}{\rho_0} - \kappa Y'' + \kappa\dfrac{\rho_0''}{\rho_0}Y &= \partial_t^6 G&\text{in}&\ \  I\times [0,T],\\
  Y&=0&\text{on}&\ \  \Gamma\times [0,T],\\
  Y|_{t=0}&=  \rho_0u_6&\text{in}&\ \  I.
  \label{Yequ}
\end{align}
where $u_6$ is given by (\ref{vt})
\subsection{Existence of a weak solution to the linear problem (\ref{fasss})-(\ref{Yequ}) by a Galerkin scheme.}
First we try to show that $\partial_t^6G$ is a function in $L^2(0,T;L^2(I))$.

\begin{equation}
  \begin{split}
    \int_0^T||\partial_t^6 G||_{L^2(I)}^2 &=\int_0^T||\partial_t^6(\dfrac{2}{\overline{\eta}'}(\dfrac{\rho_0}{\overline{\eta}'})')||_{L^2(I)}^2\\
    &\leq C\int_0^T||\rho_0\partial_t^5\overline{v}''+\partial_t^5\overline{v}'||_{L^2(I)}^2 + \text{l.o.t}\\
    &\leq C\int_0^T||\rho_0\partial_t^5 \overline{v}||_{H^2(I)}^2 + C\int_0^T||\partial_t^5 \overline{v}||_{H^1(I)}^2 +\text{l.o.t}\\
    &\leq CP(||\overline{v}||_{\mathcal{X}_T}^2)
\end{split}
  \label{Gestim}
\end{equation}

Now we begin our Galerkin scheme. Let $\{e_n\}_{n \in \mathbb{N}}$ denote a Hilbert basis of $H_0^1(I)$. Such a choice of basis is indeed possible as we can take for instance the eigenfunctions of the Laplace operator on $I$ with vanishing Dirichlet boundary conditions. We then define the Galerkin approximation at order $n \geq 1$ of (\ref{Yequ}) as  $Y_n =\sum_{i=0}^n \lambda_i^n(t)e_i$,  with $\lambda_i^n(t)$ being the solution of the ODE system:
\begin{equation}
  \begin{split}
       (\dfrac{Y_{nt}}{\rho_0},e_k)_{L^2(I)} + (\kappa Y_n', e_k')_{L^2(I)} + (\dfrac{\rho_0''}{\rho_0}Y_n, e_k)_{L^2(I)} &= (\partial_t^6G,e_k)_{L^2(I)}\\
    \lambda_i^n(0) &= (Y_{\text{init}}, e_i)_{L^2(I)}\\ \forall k &\in {0,\dots n}
\end{split}
  \label{weakform}
\end{equation}

Since each $e_i$ is in $H_0^1(I)$, we have by high-order Hardy-type inequality (\ref{hardy}) that $\dfrac{e_i}{\rho_0} \in L^2(I)$. Therefore each integral in (\ref{weakform}) is well-defined. Furthermore, as the $\{e_i\}$ are linearly independent, so are the $\{\dfrac{e_i}{\sqrt{\rho_0}}\}$ and therefore the determinant of the matrix 
\begin{equation*}
  \bigg( (\dfrac{e_i}{\sqrt{\rho_0}},\dfrac{e_j}{\sqrt{\rho_0}} )_{L^2(I)}    \bigg)_{(i,j) \in \mathbb{N}_n=\{1,\dots,n\}}
  \label{Ystart}
\end{equation*}
is nonzero. This implies that our finite-dimensional Galerkin approximation (\ref{weakform}) is a well-defined first-order differential system of order $n+1$, which therefore has a solution on a time interval $[0, T_n]$, where $T_n$ may depend on the dimension $n$ of the Galerkin approximation.

Next we show that $T_n \geq T$, with $T$ independent of $n$.

Noticing that $Y_n$ is a linear combination of the $e_i(i \in \mathbb{N}_n)$, we have that
\begin{equation}
  (\dfrac{Y_{nt}}{\rho_0}, Y_n)_{L^2(I)}+\kappa(Y_n',Y_n')_{L^2(I)}+(\dfrac{\rho_0''}{\rho_0}Y_n, Y_n)_{L^2(I)} =(\partial_t^6 G, Y_n)_{L^2(I)}
  \label{weak}
\end{equation}

Hence, we have
\begin{equation}
  \begin{split}
  &\ \ \ \dfrac{1}{2}\dfrac{d}{dt}\int_I\dfrac{Y_n^2}{\rho_0}-\kappa||\rho_0''||_{L^{\infty}}\int_I\dfrac{Y_n^2}{\rho_0} +\kappa\int_I Y_n'^2 \\
&\leq ||\partial_t^6G||_{L^2(I)}^2+||\dfrac{Y_n}{\sqrt{\rho_0}}||_{L^2(I)}^2||\rho_0||_{L^{\infty}(I)}
\end{split}
  \label{X_est}
\end{equation}

Using Poincare inequality $||Y_n||_{L^2(I)}^2 \leq ||Y_n'||_{L^2(I)}^2$ and Gronwell inequality, then we can find $T>0$ (independent of $n$) such that:
\begin{equation}
  \sup_{t \in [0,T]}C\int_I\dfrac{Y_n^2}{\rho_0} +\kappa\int_0^T||Y_n||_{H^1(I)}^2 \leq \int_0^T ||\partial_t^6G||_{L^2(I)}^2 + ||\dfrac{\rho_0u_6}{\sqrt{\rho_0}}||_{L^2(I)}^2
  \label{esti_1}
\end{equation}
noticing (\ref{Gestim}) and the fact $\overline{v} \in \mathcal{C}_{T}(M)$,
\begin{equation}
  \sup_{t \in [0,T]}C\int_I\dfrac{Y_n^2}{\rho_0} +\kappa\int_0^T||Y_n||_{H^1(I)}^2 \leq \mathcal{N}_0 + CP(||\overline{v}||_{\mathcal{X}_{T}}) 
  \label{}
\end{equation}
where $\mathcal{N}_0$ is defined in (\ref{N0}). Thus, there exists a subsequence of $(Y_n)$ which converges weakly to some $Y \in L^2(0,T;H_0^1(I))$, which satisfies
\begin{equation}
  \sup_{t \in [0,T]}C\int_I\dfrac{Y^2}{\rho_0} +\kappa\int_0^T||Y||_{H^1(I)}^2 \leq \mathcal{N}_0 + CP(||\overline{v}||_{\mathcal{X}_{T}}) 
  \label{Y_H}
\end{equation}

Now take the limit $n \to \infty$ in (\ref{weakform}), we have
\begin{equation}
 (\dfrac{Y_{t}}{\rho_0}, e_k)_{L^2(I)}+\kappa(Y',e_k')_{L^2(I)}+(\dfrac{\rho_0''}{\rho_0}Y, e_k)_{L^2(I)} =(\partial_t^6 G, e_k)_{L^2(I)}
\end{equation}
for every $k$.

Hence, (\ref{Yequ}) is satisfied in the sense of distributions, and that
\begin{equation}
  \dfrac{Y_t}{\rho_0} \in L^2(0,T;H^{-1}(I))
  \label{Yend}
\end{equation}

Now we define
\begin{equation}
  Z = \int_0^t Y(,\tau)d\tau + \rho_0u_5,
  \label{zd}
\end{equation}
\begin{equation}
  W = \int_0^t Z(.,\tau)d\tau + \rho_0u_4,
  \label{Wdef}
\end{equation}
and
\begin{equation}
  X = \int_0^t\int_0^{t_1}\int_0^{t_2}\int_0^{t_3}\int_0^{t_4}Z(,\tau)d\tau dt_4 dt_3 dt_2 dt_1 + \sum_{i=0}^{5} \dfrac{\rho_0 u_i t^{i}}{i!}.
  \label{531}
\end{equation}

We then see that $X \in W^{6,2}([0,T];H_0^1(I))$ is a solution of (\ref{stewewe})--(\ref{Xequ}), with $\partial^6_t X= Y$.

\subsection{Improving Space Regularity}
In order to prove that $v \in \mathcal{C}_{T}(M)$ and then obtain a fixed point for the map $\Theta: \overline{v} \mapsto v$, we need to establish better
space regularity for $Z$, and hence $X$ and $v$.

As $Z$ defined in (\ref{zd}), then $Z$ satisfies the following equation:
\begin{equation}
\dfrac{Z_t}{\rho_0} - \kappa Z'' +\kappa \dfrac{\rho_0''}{\rho_0}Z = \partial_t^5 G.
\label{Zequ}
\end{equation}

With high-order Hardy-type inequality, we have
\begin{equation}
  \begin{split}
  \kappa||Z''||_{L^2(I)} &\leq ||\dfrac{Z_t}{\rho_0}||_{L^2(I)}+||\dfrac{\rho_0''}{\rho_0}Z||_{L^2(I)} + ||\partial_t^5 G||_{L^2(I)} \\
  &\leq ||Y||_{H^1(I)}+||\rho_0''||_{L^{\infty}}||Z||_{H^1(I)} + ||\partial_t^5 G||_{L^2(I)}.
\end{split}
\end{equation}

So the regularity of $Z=\rho_0\partial_t^5v$ can be improved to $L^2(0,T;H^2(I))$, and then $v=\dfrac{X}{\rho_0}$ is well-defined and can be easily proved it is a solution to (\ref{linear}).

Furthermore, as $W$ defined in (\ref{Wdef}), we can see that $W= \rho_0 \partial_t^4 v$ and $W \in L^2(0,T;H^2(I))$. And we have a similarly estimate:
\begin{equation}
  \begin{split}
    \kappa||W''||_{H^1(I)} &\leq ||\dfrac{W_t}{\rho_0}||_{H^1(I)}+||\dfrac{\rho_0''}{\rho_0}W||_{H^1(I)} + ||\partial_t^4 G||_{H^1(I)} \\
  &\leq ||Z||_{H^2(I)}+||\rho_0''||_{L^{\infty}}||W||_{H^2(I)} + ||\partial_t^4 G||_{H^1(I)}.
  \end{split}
  \label{}
\end{equation}

Hence that $\rho_0\partial_t^4 v \in L^2(0,T;H^3(I))$ and $\partial_t^4 v \in L^2(0,T;H^2(I))$, and we have $v \in \mathcal{X}_T$.

\subsection{The existence of a fixed-point}
First it is clear that there is only one solution $v \in L^2(0,T;H^2(I))$ of (\ref{linear}) with $v(0)=u_0$, since if we denote by $\omega$ another solution with the same regularity, then the difference $\delta v = v -\omega$ satisfies $\delta v(\cdot,0)=0$ and $\rho_0 \delta v_t - \kappa[\rho_0^2\delta v']'=0$, which implies
\begin{equation}
  \dfrac{1}{2}\dfrac{d}{dt}\int_I\rho_0\delta v^2 +\kappa\int_I\rho_0\delta v'^2 = 0
  \label{}
\end{equation}
which together with $\delta v(\cdot,0)=0$ implies $\delta v =0$. So the mapping $\overline{v} \to v$ is well defined.

Now we will prove $v \in \mathcal{C}_{T}(M)$ when $T$ is sufficiently small.

First, we need to re-estimate $L^2(0,T;H^2(I))$-norm of $Z = \rho_0 \partial_t^5 v$. Like (\ref{X_est}), we can easily have the following:
\begin{equation}
  \dfrac{1}{2}\dfrac{d}{dt}\int_I\dfrac{Y^2}{\rho_0}-\kappa||\rho_0''||_{L^{\infty}}\int_I\dfrac{Y^2}{\rho_0} +\kappa\int_I Y'^2 \leq \big|(\partial_t^6G,Y)_{L^2(I)}\big|,
  \label{gop}
\end{equation}
and
\begin{equation}
  \sup_{t \in [0,T]}C\int_I\dfrac{Y^2}{\rho_0} +2\kappa\int_0^T||Y||_{H^1(I)}^2 \leq \int_0^T \big|(\partial_t^6G,Y)_{L^2(I)}\big| + ||\dfrac{Y_{\text{init}}}{\sqrt{\rho_0}}||_{L^2(I)}^2
  \label{Yest_2}
\end{equation}

Since
\begin{equation}
  \begin{split}
 &\ \ \ \ \int_0^T \big|(\partial_t^6G,Y)_{L^2(I)}\big| \\&=\int_0^T\big|(\sqrt{\rho_0}\partial_t^6G,\dfrac{Y}{\sqrt{\rho_0}})_{L^2(I)}\big|\\& \leq \int_0^T ||\sqrt{\rho_0}\partial_t^6 G||_{L^2(I)}||\dfrac{Y}{\sqrt{\rho_0}}||_{L^2(I)}
  \\&\leq \sup_{t \in [0,T]}||\dfrac{Y}{\sqrt{\rho_0}}||_{L^2(I)}(\int_0^T 1^2)^{1/2}(\int_0^T ||\sqrt{\rho_0}\partial_t^6 G||_{L^2(I)}^2)^{1/2}\\
  \\&\leq CTP(||\overline{v}||_{\mathcal{X}_T}^2) + CT \sup_{t \in [0,T]}||\dfrac{Y}{\sqrt{\rho_0}}||_{L^2(I)}^2,
\end{split}
  \label{}
\end{equation}
so when $T$ is sufficiently small, we can get
\begin{equation}
  \sup_{t \in [0,T]}C\int_I\dfrac{Y^2}{\rho_0} +2\kappa\int_0^T||Y||_{H^1(I)}^2 \leq \mathcal{N}_0 +  CTP(||\overline{v}||_{\mathcal{X}_T}^2).
  \label{Yest}
\end{equation}

Considering (\ref{Zequ}), and using high-order Hardy-type inequality (\ref{hardy}) and the estimate (\ref{Yest}), we have
\begin{equation}
  \begin{split}
&\ \ \ \  C\int_0^T ||Z''||_{L^2(I)}^2 \\&\leq \int_0^T||\dfrac{Z_t}{\rho_0}||_{L^2(I)}^2+\int_0^T||\dfrac{\rho_0''}{\rho_0}Z||_{L^2(I)}^2 + \int_0^T||\partial_t^5 G||_{L^2(I)}^2\\
  &\leq \int_0^T||Y||_{H^1(I)}^2 + \int_0^T||\rho_0''||_{L^{\infty}}||Z||_{H^1(I)}^2+\mathcal{N}_0+CTP(||\overline{v}||_{\mathcal{X}_T}^2)\\
  &\leq \mathcal{N}_0+CTP(||\overline{v}||_{\mathcal{X}_T}^2).
  \end{split}
\end{equation}

This implies
\begin{equation}
  \begin{split}
  ||\rho_0\partial_t^5 v||_{L^2(0,T;H^2(I))}^2 &\leq  \mathcal{N}_0+CTP(||\overline{v}||_{\mathcal{X}_T}^2)\\
  ||\partial_t^5 v||_{L^2(0,T;H^1(I))}^2 &\leq  \mathcal{N}_0+CTP(||\overline{v}||_{\mathcal{X}_T}^2)
\end{split}
  \label{}
\end{equation}

The second inequality is following by using the high-order Hardy-type inequality. The left part of $\mathcal{X}_T$ norm can be estimated in almost the same way.

So finally we get
\begin{equation}
  ||v||_{\mathcal{X}_T}^2 \leq \mathcal{N}_0+CTP(||\overline{v}||_{\mathcal{X}_T}^2)
  \label{kkkkkkkk}
\end{equation}

Take 
\begin{equation}
  T\leq \dfrac{\mathcal{N}_0}{C P(M)},
  \label{Time}
\end{equation} we have $||v||_{\mathcal{X}_T}^2 \leq 2\mathcal{N}_0$. Let us fix $M=2\mathcal{N}_0$, then $v \in C_{T}(M)$.

Now we have the mapping $\Theta: \overline{v} \to v$ is actually from $\mathcal{C}_T(M)$ into itself for $T=T_{\kappa}$ satisfying (\ref{Time}). Then, we can get a sequence of functions $v^{(n)} \in \mathcal{C}_T(M)$, where $v^{(n+1)} = \Theta(v^{(n)})$. It is obvious that $v^{(n)}$ converges weakly in $\mathcal{X}_T$. Furthermore, we have the following lemma which will show that $\rho_0v^{(n)}$ converges strongly in $L^2(0,T;H^1(I))$ and hence $v^{(n)}$ converges strongly in $L^2(0,T;L^2(I))$, which will lead a fixed-point to system (\ref{appro}).
\begin{lemma}
For the sequence of functions $v^{(n)}$ we defined before, we have: 
\begin{equation}
\begin{split}
  &\ \ \ \ ||\rho_0(v^{(n+2)}-v^{(n+1)})||_{L^2(0,T;H^1(I))}^2\\& \leq CTP(||\rho_0(v^{(n+1)}-v^{(n)})||_{L^2(0,T;H^1(I)}^2).
\end{split}
  \label{}
\end{equation}
\end{lemma}

\begin{proof}
 It is clear that $v^{(n+2)}-v^{(n+1)}$ satisfies the equation:
  \begin{equation}
    \begin{split}
\label{v1v2}
    \rho_0 (v^{(n+2)}-v^{(n+1)})_t - \kappa[\rho_0^2(v^{(n+2)}-v^{(n+1)})']' &= \rho_0[G(v^{(n+1)})-G(v^{(n)})],\\
    (v^{(n+2)}-v^{(n+1)})|_{t=0}&=0.
  \end{split}
  \end{equation}

Let $U=\rho_0 (v^{(n+2)}-v^{(n+1)})$, similar as (\ref{gop}), we have:
\begin{equation}
\begin{split}
  &\ \ \ \ \dfrac{1}{2}\dfrac{d}{dt}\int_I\dfrac{U^2}{\rho_0}-\kappa||\rho_0''||_{L^{\infty}}\int_I\dfrac{U^2}{\rho_0} +\kappa\int_I U'^2 \\&\leq \big|([G(v^{(n+1)})-G(v^{(n)})],U)_{L^2(I)}\big|\\
&=\big|(\rho_0\dfrac{(\eta^{(n+1)})'^2-(\eta^{(n)})'^2}{(\eta^{(n+1)})'^2(\eta^{(n)})'^2},\rho_0(v^{(n+2)}-v^{(n+1)})')_{L^2(I)}\big|\\
&\leq C_{\delta}||\int_0^t \rho_0(v^{(n+1)}-v^{(n)})'||_{L^2}^2+\delta||\rho_0(v^{(n+2)}-v^{(n+1)})'||_{L^2}^2\\
&\leq C_{\delta}||\int_0^t \rho_0(v^{(n+1)}-v^{(n)})'||_{L^2}^2+\delta C||U||_{H^1}^2.
\end{split}
  \label{}
\end{equation}

Then choose $\delta$ small enough, and using Poincare inequality, Gronwell inequality and high order Hardy type inequality, we finally have

\begin{equation}
\begin{split}
\dfrac{\kappa}{2}\int_0^T||U||_{H^1}^2 &\leq CTP(||\rho_0(v^{(n+1)}-v^{(n)})'||_{L^2(0,T;L^2(I))}^2)\\&\leq 
CTP(||\rho_0(v^{(n+1)}-v^{(n)})||_{L^2(0,T;H^1(I))}^2).
\end{split}
\end{equation}

\end{proof}

Thereby, we prove the following Theorem:
\begin{theorem}
  If the initial data is smooth, then there exists a unique solution $v_{\kappa} \in \mathcal{X}_T$ to the degenerate parabolic $\kappa$-problem (\ref{appro}) for sufficiently small $T$.
  \label{solution}
\end{theorem}

\section{Asymptotic estimates for $v_{\kappa}$ independent of $\kappa$}
\label{ee}
Our objective in this section is to show that the higher-order energy function E defined in (\ref{energy}) satisfies the inequality
\begin{equation}
  \sup_{t \in [0,T]} E(t) \leq M_0 +CTP(\sup_{t \in [0,T]} E(t))
  \label{Ein}
\end{equation}
where $P$ denotes a polynomial function, and for $T >0$ taken sufficiently small, with $M_0$ being a constant depending only on the initial data. The norms in $E$ are for solutions $v_{\kappa}$ to our degenerate parabolic $\kappa$-problem (\ref{appro}).

According to (\ref{solution}), $v_{\kappa} \in \mathcal{X}_{T_{\kappa}}$ with the additional bound $||\partial_t^4 v_{\kappa}||_{L^2(0,T_{\kappa})} < \infty$. As such, the energy function $E$ is continuous with respect to $t$, and the inequality (\ref{Ein}) would thus establish a time interval of existence and bound which are both independent of $\kappa$. For the sake of notational convenience, we shall denote $v_{\kappa}$ by $v$. We will generally follow the computation in [\cite{DS_2009}, sec.6].

\subsection{A $\kappa$-independent energy estimate on the $\partial_t^5$-problem}
Our starting point shall be the fifth time differentiated problem of (\ref{tytyt}) for which we have, by naturally using $\partial_t^5 v \in L^2(0,T_{\kappa};H^1(I))$ (since $v \in \mathcal{X}_{T_{\kappa}}$) as a test function, the following identity:
\begin{equation}
  \underbrace{\dfrac{1}{2}\dfrac{d}{dt}\int_I \rho_0|\partial_t^5 v|^2}_{\mathcal{I}_1} - \underbrace{\int_I \partial_t^5[\dfrac{\rho_0^2}{\eta'^2}]\partial_t^5 v'}_{\mathcal{I}_2} + \underbrace{\kappa \int_I \rho_0^2 (\partial_t^5 v')^2}_{\mathcal{I}_3} =0.
  \label{energy_1}
\end{equation}
Noticing the fact that $\partial_t^6 v \in L^2(0,T_{\kappa};L^2(I))$, which follows from (\ref{Ydefi}), (\ref{Y_H}) and high-order Hardy-type inequality, (\ref{energy_1}) is well-defined.
Upon integration in time, both the terms $\mathcal{I}_1$ and $\mathcal{I}_3$ provide sign-definite energy contributions, so we focus our attention on the nonlinear estimates required of the term $\mathcal{I}_2$.

We see that 
\begin{equation}
  \begin{split}
    -\mathcal{I}_2 =& 2 \int_I \partial_t^4 v' [\dfrac{\rho_0^2}{\eta'^3}]\partial_t^5 v'- \sum_{\alpha=1}^4 b_{\alpha}\int_I \partial_t^{\alpha}\dfrac{1}{\eta'^3}\partial_t^{4-\alpha}v'\rho_0^2\partial_t^5 v'\\
    =&\dfrac{d}{dt}\int_I(\partial_t^4 v')^2\dfrac{\rho_0^2}{\eta'^3} + 3\int_I(\partial_t^4 v')^2 v' \dfrac{\rho_0^2}{\eta'^4} - \sum_{\alpha=1}^4 b_{\alpha}\int_I \partial_t^{\alpha}\dfrac{1}{\eta'^3}\partial_t^{4-\alpha}v'\rho_0^2\partial_t^5 v'.
\end{split}
  \label{}
\end{equation}

Hence integrating (\ref{energy_1}) from $0$ to $t \in [0, T_{\kappa}]$, we find that
\begin{equation}
  \begin{split}
  &\dfrac{1}{2}\int_I \rho_0\partial_t^5 v^2(t) + \int_I(\partial_t^4 v')^2\dfrac{\rho_0^2}{\eta'^3}(t) + \kappa\int_0^t\int_I \rho_0^2 (\partial_t^5 v')^2\\
  =& \dfrac{1}{2}\int_I \rho_0\partial_t^5 v^2(0) + \int_I(\partial_t^4 v')^2\dfrac{\rho_0^2}{\eta'^3}(0) - 3\int_0^t\int_I(\partial_t^4 v')^2 v' \dfrac{\rho_0^2}{\eta'^4}\\& + \sum_{\alpha=1}^4 b_{\alpha}\int_I \partial_t^{\alpha}\dfrac{1}{\eta'^3}\partial_t^{4-\alpha}v'\rho_0^2\partial_t^5 v'.
\end{split}
  \label{error_term}
\end{equation}
We next show that all of the error terms can be bounded by $CtP(\sup_{[0,t]}E)$. First, it is clear that
\begin{equation}
  \begin{split}
  -3\int_0^t\int_I(\partial_t^4 v')^2 v' \dfrac{\rho_0^2}{\eta'^4} &\leq C \int_0^t ||v'||_{L^{\infty}} ||\rho_0 \partial_t^4 v'||_{L^2}^2
  \\&\leq C \int_0^t ||v||_{H^2} (||\rho_0 \partial_t^4 v||_{H^1}^2+||\partial_t^4v||_{L^2}^2)
  \\&\leq CtP(\sup_{[0,t]}E).
  \end{split}
  \label{}
\end{equation}

Then using integration-by-parts in time, we have that
\begin{equation}
\begin{split}
&\ \   \int_0^t\int_I\sum_{\alpha=1}^4 b_{\alpha} \partial_t^{\alpha}\dfrac{1}{\eta'^3}\partial_t^{4-\alpha}v'\rho_0^2\partial_t^5 v' \\&= \underbrace{\int_0^t\int_I(\sum_{\alpha=1}^4 b_{\alpha} \partial_t^{\alpha}\dfrac{1}{\eta'^3}\partial_t^{4-\alpha}v')_t\rho_0^2\partial_t^4 v'}_{J}+\int_I\sum_{\alpha=1}^4 b_{\alpha} \partial_t^{\alpha}\dfrac{1}{\eta'^3}\partial_t^{4-\alpha}v'\rho_0^2\partial_t^4 v'\bigg \arrowvert_0^t.
\end{split}
  \label{error_2}
\end{equation}

The term J can be written under the form of the sum of space-time integrals of the following types:
\begin{equation}
  \begin{split}
    J_1&=\int_0^t\int_I\rho_0\partial_t^4 v' v' R(\eta')\rho_0\partial_t^4 v',\\
    J_2&=\int_0^t\int_I\rho_0\partial_t^3 v' (v')^2 R(\eta')\rho_0\partial_t^4 v',\\
J_3&=\int_0^t\int_I\rho_0\partial_t^3 v' \partial_t v' R(\eta')\rho_0\partial_t^4 v',\\
    J_4&=\int_0^t\int_I\rho_0\partial_t^2 v' \partial_t v' v' R(\eta')\rho_0\partial_t^4 v',\\
J_5&=\int_0^t\int_I\rho_0(\partial_t^2 v')^2 R(\eta')\rho_0\partial_t^4 v',\\
J_6&=\int_0^t\int_I\rho_0\partial_t^2 v'(v')^3 R(\eta')\rho_0\partial_t^4 v',\\
    J_7&=\int_0^t\int_I\rho_0(\partial_t v')^3 R(\eta')\rho_0\partial_t^4 v',\\
J_8&=\int_0^t\int_I\rho_0(\partial_t v')^2 (v')^2 R(\eta')\rho_0\partial_t^4 v'.\\
  \end{split}
  \label{}
\end{equation}
where $R(\eta')$ denotes a rational function of $\eta'$.

We first immediately see that
\begin{equation}
  |J_1| \leq C \int_0^t ||v'||_{L^{\infty}} ||\rho_0 \partial_t^4 v'||_{L^2}^2 ||R(\eta')||_{L^{\infty}}  \leq CtP(\sup_{[0,t]}E)  \label{i1}.
\end{equation}

Next, we have that
\begin{equation}
  \begin{split}
    |J_3| &\leq C \int_0^t ||\rho_0\partial_t^3 v'||_{L^4}||\partial_t v'||_{L^{4}}||R(\eta')||_{L^{\infty}}||\rho_0\partial_t^4 v'||_{L^2}\\
  &\leq CtP(\sup_{[0,t]}E),
  \end{split}
  \label{i2}
\end{equation}
and
\begin{equation}
  \begin{split}
    |J_7| &\leq C \int_0^t ||\partial_t v'||_{L^6}^3||R(\eta')||_{L^{\infty}}||\rho_0\partial_t^4 v'||_{L^2}\\
  &\leq CtP(\sup_{[0,t]}E),
  \end{split}
  \label{i3}
\end{equation}
where we used Sobolev embedding inequalities in 1-D, $||\cdot||_{L^{\infty}} \leq C_p ||\cdot||_{H^{1}}$  and $||\cdot||_{L^{p}} \leq C_p ||\cdot||_{H^{\frac{1}{2}}}$, for all $1< p < \infty$.

$J_2$, $J_4$, $J_5$, $J_6$ and $J_8$ can be estimated almost in the same way.

The term $\int_I b_{\alpha} \partial_t^{\alpha}\dfrac{1}{\eta'^3}\partial_t^{4-\alpha}v'\rho_0^2\partial_t^4 v'\bigg \arrowvert_0^t$
can be estimated by 
\begin{equation}E^{\frac{1}{2}}(M_0+CtP(\sup_{t\in[0,T]}E))\end{equation} in the similar way by using the fundamental theorem of calculus.

Therefore, using Young's inequality, we have 
\begin{equation}
  \dfrac{1}{2}\int_I \rho_0\partial_t^5 v^2(t) + \int_I(\partial_t^4 v')^2\dfrac{\rho_0^2}{\eta'^3}(t) 
+ \kappa\int_0^t\int_I \rho_0^2 (\partial_t^5 v')^2 \leq M_0 +CtP(\sup_{[0,t]}E),
  \label{}
\end{equation}
and thus, using the fundamental theorem of calculus,
\begin{equation}
  \begin{split}
 &\ \ \ \ \dfrac{1}{2}\int_I \rho_0\partial_t^5 v^2(t) + \int_I(\rho_0\partial_t^4 v')^2(t)
+ \kappa\int_0^t\int_I \rho_0^2 (\partial_t^5 v')^2\\&\leq M_0 + CtP(\sup_{[0,t]}E).
  \end{split}
  \label{energyest}
\end{equation}

\subsection{Elliptic and Hardy-type estimates for $\partial_t^2 v(t)$ and $v(t)$}
Having obtained the energy estimate (\ref{energyest}) for the $\partial_t^5$-problem, we can begin our bootstrapping argument. We now consider the $\dfrac{1}{\rho_0}\partial_t^3$-problem of (\ref{tytyt})
\begin{equation}
  \dfrac{1}{\rho_0}[\partial_t^3\dfrac{\rho_0^2}{\eta'^2}]' -\dfrac{\kappa}{\rho_0}[\rho_0^2\partial_t^3 v']' = -\partial_t^4v,
  \label{}
\end{equation}
which can be written as 
\begin{equation}
  -\dfrac{2}{\rho_0}[\dfrac{\rho_0^2\partial_t^2v'}{\eta'^3}]' - \dfrac{\kappa}{\rho_0}[\rho_0^2\partial_t^3 v']' = -\partial_t^4v + \dfrac{c_1}{\rho_0}[\dfrac{\rho_0^2\partial_t v'v'}{\eta'^4}]'+\dfrac{c_2}{\rho_0}[\dfrac{\rho_0^2v'^3}{\eta'^5}]',
  \label{}
\end{equation}
and finally be rewritten as the following identity:
\begin{equation}
  \begin{split}
  -\dfrac{2}{\rho_0}[\rho_0^2\partial_t^2v']' - \dfrac{\kappa}{\rho_0}[\rho_0^2\partial_t^3 v']' = &-\rho_0\partial_t^4v + \dfrac{c_1}{\rho_0}[\dfrac{\rho_0^2\partial_t v'v'}{\eta'^4}]'+\dfrac{c_2}{\rho_0}[\dfrac{\rho_0^2v'^3}{\eta'^5}]'\\&-2\dfrac{1}{\rho_0}[\rho_0^2\partial_t^2v']'(1-\dfrac{1}{\eta'^3})-6\rho_0\partial_t^2v'\dfrac{\eta''}{\eta'^4}.
\end{split}
  \label{aaa}
\end{equation}
Here, $c_1$ and $c_2$ are constants whose exact values are not important.

Therefore, using Lemma (\ref{para_lema}) and the fundamental theorem of calculus for the fourth term on the right-hand side of (\ref{aaa}), we obtain that for any $t \in[0,T_{\kappa}]$:
\begin{equation}
  \begin{split}
 &   \sup_{[0,t]}||\dfrac{2}{\rho_0}[\rho_0^2\partial_t^2v']'||_{L^2} \\ \leq &\sup_{[0,t]}||\partial_t^4v||_{L^2} + \sup_{[0,t]}||\dfrac{c_1}{\rho_0}[\dfrac{\rho_0^2\partial_t v'v'}{\eta'^4}]'||_{L^2}\\&+\sup_{[0,t]}||\dfrac{c_2}{\rho_0}[\dfrac{\rho_0^2v'^3}{\eta'^5}]'||_{L^2}+\sup_{[0,t]}||\dfrac{2}{\rho_0}[\rho_0^2\partial_t^2v']'||_{L^2}||3\int_0^{.}\dfrac{v'}{\eta'^4}||_{L^{\infty}}\\&+6\sup_{[0,t]}||\rho_0\partial_t^2v'\dfrac{\eta''}{\eta'^4}||_{L^2}.
\end{split}
  \label{bbb}
\end{equation}
We next estimate each term on the right hand side of (\ref{bbb}). For the first term, we will use our estimate (\ref{energyest}) from which we infer for each $t \in [0,T_{\kappa}]$:
\begin{equation}
  \int_I \rho_0^2[|\partial_t^4 v|^2+|\partial_t^4 v'|^2](t) \leq M_0 +CtP(\sup_{[0,t]}E).
  \label{ccc}
\end{equation}
Note that the first term of left-hand side of (\ref{ccc}) comes from the first term of (\ref{energyest}), together with the fact that $\partial_t^4 v(t,x) = \partial_t^4v(x,0)+\int_0^t \partial_t^5v(.,x)$. Therefore, the Sobolev weighted embedding estimate (\ref{emb}) provides us with the following estimate:
\begin{equation}
  \int_I|\partial_t^4 v|^2(t) \leq M_0+CtP(\sup_{[0,t]}E).
  \label{uu}
\end{equation}

The remaining terms will be estimated by simply using the definition of the energy function $E$. 

For the second term, we have  that:
\begin{equation}
  \begin{split}
&\ \ \ \    ||\dfrac{1}{\rho_0}[\dfrac{\rho_0^2\partial_tv'v'}{\eta'^4}]'||_{L^2} \\&\leq ||(\rho_0\partial_tv')'||_{L^2}||\dfrac{v'}{\eta'^4}||_{L^{\infty}} + ||\partial_tv'[\dfrac{\rho_0v'}{\eta'^4}]'||_{L^2}\\
    &\leq C||(\rho_0v_t')'||_{L^2}||v'||_{L^{\infty}}+||v_t'[\dfrac{\rho_0'v'}{\eta'^4}]||_{L^2} \\&\ \ + ||v_t'[\dfrac{\rho_0v''}{\eta'^4}]||_{L^2}+4||v_t'[\dfrac{\rho_0v'\eta''}{\eta'^5}]||_{L^2}\\
    &\leq C||(\rho_0u_1')'+\int_0^.(\rho_0v_{tt}')'||_{L^2}||v'||_{H^{1}}^{\frac{1}{2}}||u_0'+\int_0^.v_t'||_{H^{\frac{1}{2}}}^{\frac{1}{2}} \\
    &\ \ \ \ +C||u_1'+\int_0^.v_{tt}'||_{L^2}||v'||_{H^{1}}^{\frac{1}{2}}||u_0'+\int_0^.v_t'||_{H^{\frac{1}{2}}}^{\frac{1}{2}}\\&\ \ \ \ +C||u_1'+\int_0^.v_{tt}'||_{L^2}||\rho_0v''||_{H^{1}}^{\frac{3}{4}}||\sqrt{\rho_0}u_0'' + \int_0^.\sqrt{\rho_0}v_t''||_{L^2}^{\frac{1}{4}}\\&\ \ \ \ +C||v'||_{H^{1}}^{\frac{1}{2}}||u_0'+\int_0^.v_t'||_{H^{\frac{1}{2}}}^{\frac{1}{2}}||\int_0^.v''||_{L^2}||\rho_0u_1'+\int_0^.\rho_0v_{tt}'||_{H^{1}}\\
     &\leq C \sup_{[0,t]}E^{\frac{3}{8}}(M_0+tP(\sup_{[0,t]}E)).
\end{split}
\end{equation}

Using the definition of $E$, then for any $t \in [0,T_{\kappa}]$, we have
\begin{equation}
  \sup_{[0,t]}||\dfrac{1}{\rho_0}[\dfrac{\rho_0^2\partial_tv'v'}{\eta'^4}]'||_{L^2} \leq C \sup_{[0,t]}E^{\frac{3}{8}}(M_0+tP(\sup_{[0,t]}E)).
  \label{b1}
\end{equation}
For the third term, we see that
\begin{equation}
  \begin{split}
&\ \ \ \    ||\dfrac{1}{\rho_0}[\dfrac{\rho_0^2v'^3}{\eta'^5}]'||_{L^2} \\&\leq 2||\dfrac{\rho_0'v'^3}{\eta'^5}||_{L^2} +3||v'[\dfrac{\rho_0v''v'}{\eta'^5}]||_{L^2}+5||v'[\dfrac{\rho_0v'^2\eta''}{\eta'^6}]||_{L^2}\\
    &\leq C ||v'||_{H^{1}}^{\frac{1}{2}}||u_0'+\int_0^.v_t'||_{H^{\frac{1}{2}}}^{\frac{1}{2}}||u_0'+\int_0^.v_t'||_{H^{\frac{1}{2}}}^2\\&\ \ \ \ +C||\rho_0v''||_{H^{1}}^{\frac{3}{4}}||\sqrt{\rho_0}u_0'' + \int_0^.\sqrt{\rho_0}v_t''||_{L^2}^{\frac{1}{4}}||u_0'+\int_0^.v_t'||_{H^{\frac{1}{2}}}^2\\
    &\ \ \ \ +C||v'||_{H^{1}}^{\frac{1}{2}}||u_0'+\int_0^.v_t'||_{H^{\frac{1}{2}}}^{\frac{1}{2}}||u_0'+\int_0^.v_t'||_{H^{\frac{1}{2}}}^2||\int_0^.\rho_0v''||_{H^{1}},
  \end{split}
  \label{b2}
\end{equation}
where we used the fact that $||\cdot||_{L^4} \leq C_p||\cdot||_{H^{\frac{1}{2}}}$. Again, using the definition of $E$, the previous inequality provides us for any $t \in [0,T_{\kappa}]$ with
\begin{equation}
  \sup_{[0,t]}||\dfrac{1}{\rho_0}[\dfrac{\rho_0^2v'^3}{\eta'^5}]'||_{L^2} \leq C \sup_{[0,t]}E^{\frac{3}{8}}(M_0+tP(\sup_{[0,t]}E)).
  \label{}
\end{equation}
For the fourth term, we see that
\begin{equation}
  \begin{split}
  ||\dfrac{2}{\rho_0}[\rho_0^2\partial_t^2v']'||_{L^2}||3\int_0^.\dfrac{v'}{\eta'^4}||_{L^{\infty}}(t) &\leq C[||\rho_0\partial_t^2v''||_{L^2}+||\partial_tv'||_{L^2}]t\sup_{[0,t]}||v||_{H^2}\\
  &\leq CtP(\sup_{[0,t]}E).
\end{split}
  \label{}
\end{equation}
Similarly, the fifth term yields the following estimate:
\begin{equation}
  \begin{split}
    ||\rho_0\partial_t^2v'\dfrac{\eta''}{\eta'^4}||_{L^2}(t) &\leq C||\rho_0\partial_t^2v'||_{L^{\infty}}||\eta''||_{L^2}\\
    &\leq C||\rho_0\partial_t^2v'||_{H^{1}}||\int_0^. v''||_{L^2}\\
    &\leq CtP(\sup_{[0,t]}E).
  \end{split}
  \label{b4}
\end{equation}
Combining the estimates (\ref{uu})-(\ref{b4}), we obtain the inequality
\begin{equation}
  \sup_{[0,t]}||\dfrac{2}{\rho_0}[\rho_0^2\partial_t^2v']'||_{L^2} \leq M_0+ CtP(\sup_{[0,t]}E) +C \sup_{[0,t]}E^{\frac{3}{8}}(M_0+tP(\sup_{[0,t]}E)).
  \label{gg}
\end{equation}
We recall that the solution $v$ to our parabolic $\kappa$-problem is in $\mathcal{X}_{T_{\kappa}}$, so for any $t \in [0,T_{\kappa}], \partial_t^2v \in H^2(I)$. Notice that
\begin{equation}
  \dfrac{1}{\rho_0}[\rho_0^2\partial_t^2v']' = \rho_0\partial_t^2v''+2\rho_0'\partial_t^2v'
  \label{}
\end{equation}
so (\ref{gg}) is equivalent to
\begin{equation}
  \sup_{[0,t]}||\rho_0\partial_t^2v''+2\rho_0'\partial_t^2v'||_{L^2} \leq CtP(\sup_{[0,t]}E) +C \sup_{[0,t]}E^{\frac{3}{8}}(M_0+tP(\sup_{[0,t]}E))
  \label{gggg}
\end{equation}
From this inequality, we would like to conclude that both $||\partial_t^2v'||_{L^2}$ and $||\rho_0\partial_t^2v''||_{L^2}$ are bounded by the right-hand side of (\ref{gggg}); the regularity provided by solutions of the $\kappa$-problem allow us to arrive at this conclusion.

By expanding the left-hand side of (\ref{gggg}), we see that
\begin{equation}
  ||\rho_0\partial_t^2v''+2\rho_0'\partial_t^2v'||_{L^2}^2=||\rho_0\partial_t^2v''||_{L^2}^2+4||\rho_0'\partial_t^2v'||_{L^2}^2 + 4\int_I\rho_0\partial_t^2v''\rho_0'\partial_t^2v'.
  \label{qw}
\end{equation}
We notice that the cross-term (\ref{qw}) is an exact derivative with the regularity of $\partial_t^2v$ provide by our $\kappa$-problem,
\begin{equation}
  4\int_I\rho_0\partial_t^2v''\rho_0'\partial_t^2v' = 2 \int_I \rho_0\rho_0'\dfrac{\partial}{\partial_x}|\partial_t^2v'|^2.
  \label{}
\end{equation}
So that by integrating-by-parts, we find that
\begin{equation}
  4\int_I\rho_0\partial_t^2v''\rho_0'\partial_t^2v' = -2||\rho_0'\partial_t^2v'||_{L^2}^2-\int_I \rho_0\partial_t^2v'\rho_0''\partial_t^2v',
  \label{}
\end{equation}
and hence (\ref{qw}) becomes
\begin{equation}
   ||\rho_0\partial_t^2v''+2\rho_0'\partial_t^2v'||_{L^2}^2= ||\rho_0\partial_t^2v''||_{L^2}^2+2||\rho_0'\partial_t^2v'||_{L^2}^2-\int_I \rho_0\partial_t^2v'\rho_0''\partial_t^2v'.
  \label{eeeeee}
\end{equation}
Since the energy function $E$ contains $||\sqrt{\rho_0}\partial_t^3v(t)'||_{L^2(I)}$, the fundamental theorem of calculus show that
\begin{equation}
  \int_I \rho_0\partial_t^2v'\rho_0''\partial_t^2v' \leq C||\sqrt{\rho_0}u_2'+\int_0^{\cdot}\sqrt{\rho_0}\partial_t^3 v'||_{L^2}^2 \leq M_0+ CtP(\sup_{[0,t]}E). 
  \label{}
\end{equation}
Combing this inequality with (\ref{eeeeee}) and (\ref{gg}), yields
\begin{equation}
\begin{split}
  &\ \ \ \ \sup_{[0,t]}[||\rho_0\partial_t^2v''||_{L^2}^2+||\rho_0'\partial_t^2v'||_{L^2}^2] \\&\leq  M_0+CtP(\sup_{[0,t]}E) +C \sup_{[0,t]}E^{\frac{3}{4}}(M_0+tP(\sup_{[0,t]}E)).
\end{split}
  \label{}
\end{equation}
and thus 
\begin{equation}
\begin{split}
  &\ \ \ \ \sup_{[0,t]}[||\rho_0\partial_t^2v''||_{L^2}^2+||\rho_0'\partial_t^2v'||_{L^2}^2+||\rho_0\partial_t^2v'||_{L^2}^2] \\&\leq  M_0+CtP(\sup_{[0,t]}E) +C \sup_{[0,t]}E^{\frac{3}{4}}(M_0+tP(\sup_{[0,t]}E)).
\end{split}  \label{}
\end{equation}
and hence with the physical vacuum conditions of $\rho_0$ given by (\ref{cond_2}) and (\ref{cond_3}), we have that
\begin{equation}
\begin{split}
&\ \ \ \   \sup_{[0,t]}[||\rho_0\partial_t^2v''||_{L^2}^2+||\partial_t^2v'||_{L^2}^2] \\&\leq  M_0+CtP(\sup_{[0,t]}E) +C \sup_{[0,t]}E^{\frac{3}{4}}(M_0+tP(\sup_{[0,t]}E)).
\end{split}
  \label{uuuuu}
\end{equation}
which, together with (\ref{uu}), provide us with the estimate 
\begin{equation}
\begin{split}
&\ \ \ \   \sup_{[0,t]}[||\rho_0\partial_t^2v''||_{L^2}^2+||\partial_t^2v||_{H^{1}}^2] \\&\leq  M_0+CtP(\sup_{[0,t]}E) +C \sup_{[0,t]}E^{\frac{3}{4}}(M_0+tP(\sup_{[0,t]}E)).
\end{split}
  \label{est_for_vtt}
\end{equation}

By studying the $\partial_x(\dfrac{1}{\rho_0}\partial_t)$-problem of (\ref{tytyt}) in the same manner, we find that
\begin{equation}
\begin{split}  
&\ \ \ \ \sup_{[0,t]}[||\rho_0v'''||_{L^2}^2+||v||_{H^2}^2]\\ &\leq M_0+CtP(\sup_{[0,t]}E)+C \sup_{[0,t]}E^{\frac{3}{4}}(M_0+tP(\sup_{[0,t]}E)).
\end{split}
  \label{last0}
\end{equation}
\subsection{Elliptic and Hardy-type estimates for $\partial_t^3 v(t)$ and $\partial_t v(t)$}
We consider the $\dfrac{1}{\sqrt{\rho_0}}\partial_t^4$-problem of (\ref{appro}):
\begin{equation}
  \dfrac{1}{\sqrt{\rho_0}}[\partial_t^4\dfrac{\rho_0^2}{\eta'^2}]'-\dfrac{\kappa}{\sqrt{\rho_0}}[\rho_0^2\partial_t^4v']'=-\sqrt{\rho_0}\partial_t^5v,
  \label{}
\end{equation}
By employing the fundamental theorem of calculus, it can be rewritten as
\begin{equation}
  \begin{split}
  &-\dfrac{2}{\sqrt{\rho_0}}[\rho_0^2\partial_t^3v']'-\dfrac{\kappa}{\sqrt{\rho_0}}[\rho_0^2\partial_t^4v']'\\=&-\sqrt{\rho_0}\partial_t^5v+\dfrac{c_1}{\sqrt{\rho_0}}[\dfrac{\rho_0^2\partial_t^2v'v'}{\eta'^4}]'+\dfrac{c_2}{\sqrt{\rho_0}}[\dfrac{\rho_0^2\partial_tv'^2}{\eta'^5}]'\\
  &-\dfrac{2}{\sqrt{\rho_0}}[\rho_0^2\partial_t^3v']'(1-\dfrac{1}{\eta'^3})-6\rho_0^{\frac{3}{2}}\partial_t^3v'\dfrac{\eta''}{\eta'^4},
  \end{split}
  \label{}
\end{equation}
for some constants $c_1$ and $c_2$.

For any $t \in [0,T_{\kappa}]$, Lemma (\ref{para_lema}) provides the $\kappa$-independent estimate
\begin{equation}
  \begin{split}
 &\sup_{[0,t]}||\dfrac{2}{\sqrt{\rho_0}}[\rho_0^2\partial_t^3v']'||_{L^2} \\ \leq &\sup_{[0,t]}||\sqrt{\rho_0}\partial_t^5v||_{L^2}+\sup_{[0,t]}||\dfrac{c_1}{\sqrt{\rho_0}}[\dfrac{\rho_0^2\partial_t^2v'v'}{\eta'^4}]'||_{L^2}\\&+\sup_{[0,t]}||\dfrac{c_2}{\sqrt{\rho_0}}[\dfrac{\rho_0^2\partial_tv'^2}{\eta'^5}]'||_{L^2}+\sup_{[0,t]}||\dfrac{2}{\sqrt{\rho_0}}[\rho_0^2\partial_t^3v']'||_{L^2}||3\int_0^.\dfrac{v'}{\eta'^4}||_{L^{\infty}}\\&+6\sup_{[0,t]}||\rho_0^{\frac{3}{2}}\partial_t^3v'\dfrac{\eta''}{\eta'^4}||_{L^2}.
  \end{split}
  \label{ss}
\end{equation}
We estimate each term on the right hand side of (\ref{ss}).

The first term on the right-hand side is bounded by $M_0+CtP(\sup_{[0,t]}E)$ due to (\ref{energyest}).

For the second term, we have that
\begin{equation}
  \begin{split}
  &\ \ \ \   ||\dfrac{1}{\sqrt{\rho_0}}[\dfrac{\rho_0^2\partial_t^2v'v'}{\eta'^4}]'||_{L^2} \\&\leq ||\dfrac{\sqrt{\rho_0}\partial_t^2v'(\rho_0v')'}{\eta'^4}||_{L^2}+||\dfrac{\sqrt{\rho_0}v'(\rho_0\partial_t^2v')'}{\eta'^4}||_{L^2}+4||\dfrac{\sqrt{\rho_0}\partial_t^2v'\rho_0v'\eta''}{\eta'^5}]||_{L^2}\\
    &\leq C||\sqrt{\rho_0}u_2'+\int_0^.\sqrt{\rho_0}\partial_t^3v'||_{L^2}||(\rho_0v')'||_{L^{\infty}}\\
&\ \ \ \ +C||v'||_{H^{1}}^{\frac{1}{2}}||u_0'+\int_0^.v_t'||_{H^{\frac{1}{2}}}^{\frac{1}{2}}||\sqrt{\rho_0}(\rho_0\partial_t^2v')'||_{L^{2}}\\
        &\ \ \ \ +C||v'||_{H^{1}}^{\frac{1}{2}}||u_0'+\int_0^.v_t'||_{H^{\frac{1}{2}}}^{\frac{1}{2}}||\int_0^.v''||_{L^2}||\rho_0^{\frac{3}{2}}\partial_t^2v'||_{L^{\infty}}\\
    &\leq C||\sqrt{\rho_0}u_2'+\int_0^.\sqrt{\rho_0}\partial_t^3v'||_{L^2}||v'||_{H^{1}}^{\frac{1}{2}}||u_0'+\int_0^.v_t'||_{H^{\frac{1}{2}}}^{\frac{1}{2}}\\
    &\ \ \ \ +C||\sqrt{\rho_0}u_2'+\int_0^.\sqrt{\rho_0}\partial_t^3v'||_{L^2}||\rho_0v''||_{H^{1}}^{\frac{3}{4}}||\sqrt{\rho_0}u_0'+\int_0^.\sqrt{\rho_0}v_t''||_{L^2}^{\frac{1}{4}}    \\
&\ \ \ \ +C||v'||_{H^{1}}^{\frac{1}{2}}||u_0'+\int_0^.v_t'||_{H^{\frac{1}{2}}}^{\frac{1}{2}}||\sqrt{\rho_0}u_2'+\int_0^.\sqrt{\rho_0}\partial_t^3v'||_{L^2}\\
&\ \ \ \ +C||v'||_{H^{1}}^{\frac{1}{2}}||u_0'+\int_0^.v_t'||_{H^{\frac{1}{2}}}^{\frac{1}{2}}||\rho_0\sqrt{\rho_0}u_2'+\int_0^.\rho_0\sqrt{\rho_0}\partial_t^3v''||_{L^2}\\
&\ \ \ \ +C||v'||_{H^{1}}^{\frac{1}{2}}||u_0'+\int_0^.v_t'||_{H^{\frac{1}{2}}}^{\frac{1}{2}}||\int_0^.v''||_{L^2}||\rho_0\partial_t^2v'||_{H^1}.
  \end{split}
  \label{ert}
\end{equation}
where we have again used fact that $||\cdot||_{L^{\infty}}\leq C||\cdot||_{H^{\frac{3}{4}}}$. Using the definition of $E$, it shows that for any $t \in [0,T_{\kappa}]$,
\begin{equation}
  \sup_{[0,t]}||\dfrac{1}{\sqrt{\rho_0}}[\dfrac{\rho_0^2\partial_t^2v'v'}{\eta'^4}]'||_{L^2} \leq C\sup_{[0,t]}E^{\frac{3}{8}}(M_0+CtP(\sup_{[0,t]}E)).
  \label{}
\end{equation}
For the third term on the right-hand side of (\ref{ss}), we have similarly that
\begin{equation}
  \begin{split}
    &\ \ \ \ ||\dfrac{1}{\sqrt{\rho_0}}[\dfrac{\rho_0^2\partial_tv'^2}{\eta'^5}]'||_{L^2}(t)\\ &\leq 2||(\rho_0\partial_tv')'||_{L^2}||\sqrt{\rho_0}\dfrac{\partial_tv'}{\eta'^5}||_{L^{\infty}}+5||\sqrt{\rho_0}\dfrac{\partial_tv'^2}{\eta'^6}||_{L^2}||\rho_0\eta''||_{L^{\infty}}\\
    &\leq C||(\rho_0 u_1')'+\int_0^.(\rho_0\partial_t^2v')'||_{L^2}||\sqrt{\rho_0}\partial_tv'||_{L^2}^{1-\alpha}||(\sqrt{\rho_0}\partial_tv')'||_{L^{2-a}}^{\alpha}\\
    &\ \ \ \ +C||\partial_tv'||_{L^4}^2||\int_0^.(\rho_0v'')'||_{L^2}\\
    &\leq C||(\rho_0 u_1')'+\int_0^.(\rho_0\partial_t^2v')'||_{L^2}||u_1'+\int_0^.\partial_t^2v'||_{L^2}^{1-\alpha}||(\sqrt{\rho_0}\partial_tv')'||_{L^{2-a}}^{\alpha}\\
    &\ \ \ \ +C||\partial_tv'||_{H^{\frac{1}{2}}}^2||\int_0^.(\rho_0v'')'||_{L^2}.
  \end{split}
  \label{hh}
\end{equation}
where $0<a<\frac{1}{2}$ is given and $0 < \alpha=\dfrac{3-3a}{4+3a} <1$.

The only term on the right-hand side of (\ref{hh}) which is not directly contained in the definition of $E$ is $||(\sqrt{\rho_0}\partial_tv')'||_{L^{2-a}}^{\alpha}$. Then we notice that
\begin{equation}
  \begin{split}
    ||(\sqrt{\rho_0}\partial_tv')'||_{L^{2-a}} &\leq ||\dfrac{\partial_tv'}{2\sqrt{\rho_0}}||_{L^{2-a}}+||\sqrt{\rho_0}v_t''||_{L^2}\\
    &\leq ||\dfrac{1}{2\sqrt{\rho_0}}||_{L^{2-\frac{a}{2}}}||\partial_tv'||_{H^{\frac{1}{2}}}+||\sqrt{\rho_0}v_t''||_{L^2}
  \end{split}
  \label{0404}
\end{equation}
where we have used the fact that $||\cdot||_{L^p} \leq C||\cdot||_{H^{\frac{1}{2}}}$, for all $1<p<\infty$. So (\ref{hh}) and (\ref{0404}) provides us for any $t \in [0,T_{\kappa}]$ with
\begin{equation}
  \sup_{[0,t]}||\dfrac{1}{\sqrt{\rho_0}}[\dfrac{\rho_0^2\partial_tv'^2}{\eta'^5}]'||_{L^2} \leq C\sup_{[0,t]}E^{\frac{\alpha}{2}}(M_0+tP(\sup_{[0,t]}E)).
  \label{}
\end{equation}
where $0 < \alpha=\dfrac{3-3a}{4+3a} <1$.

The fourth term on the right-hand side of (\ref{ss}) is easily treated as:
\begin{equation}
  \begin{split}
   &\ \ \ \ ||\dfrac{1}{\sqrt{\rho_0}}[\rho_0^2\partial_t^3v']'||_{L^2}||\int_0^.\dfrac{v'}{\eta'^4}||_{L^{\infty}}(t)\\ &\leq C[||\rho_0^{\frac{3}{2}}\partial_t^3v''||_{L^2}+||\sqrt{\rho_0}\partial_t^3v'||_{L^2}]t\sup_{[0,t]}||v||_{H^2}\\
    &\leq CtP(\sup_{[0,t]}E).
\end{split}
  \label{}
\end{equation}
Similarly, the fifth term is estimated as follows:
\begin{equation}
  \begin{split}
    ||\rho_0^{\frac{3}{2}}\partial_t^3v'\dfrac{\eta''}{\eta'^4}||_{L^2}(t) &\leq C||\sqrt{\rho_0}\partial_t^3v'||_{L^2}||\int_0^t\rho_0v''||_{H^1}\\
    &\leq CtP(\sup_{[0,t]}E).
  \end{split}
  \label{ggggggg}
\end{equation}
Combining the estimates (\ref{ert})-(\ref{ggggggg}), we can show that
\begin{equation}
  \sup_{[0,t]}||\dfrac{1}{\sqrt{\rho_0}}[\rho_0^2\partial_t^3v']'||_{L^2} \leq M_0+CtP(\sup_{[0,t]}E)+C\sup_{[0,t]}E^{\frac{\alpha}{2}}(M_0+tP(\sup_{[0,t]}E)).
  \label{rtu}
\end{equation}
Now, since for any $t \in [0,T_{\kappa}]$, solutions to our parabolic $\kappa$-problem have the regularity $\partial_t^2v \in H^2(I)$, we integrate-by-parts:
\begin{equation}
  \begin{split}
    ||\dfrac{1}{\sqrt{\rho_0}}[\rho_0^2\partial_t^3v']'||_{L^2}^2 &=||\rho_0^{\frac{3}{2}}\partial_t^3v''||_{L^2}^2+4||\sqrt{\rho_0}\rho_0'\partial_t^3v'||_{L^2}^2+2\int_I\rho_0'\rho_0^2[|\partial_t^3v'|^2]'\\ 
    &=||\rho_0^{\frac{3}{2}}\partial_t^3v''||_{L^2}^2-2\int_I\rho_0''\rho_0^2|\partial_t^3v'|^2.
  \end{split}
  \label{zs}
\end{equation}
Combined with (\ref{rtu}), and the fact that $\rho_0\partial_t^3v'=\rho_0u_3'+\int_0^.\rho_0\partial_t^4v'$ for the second term on the right-hand side of (\ref{zs}), we find that
\begin{equation}
  \sup_{[0,t]}||\rho_0^{\frac{3}{2}}\partial_t^3v''||_{L^2}^2 \leq M_0+CtP(\sup_{[0,t]}E)+C\sup_{[0,t]}E^{\alpha}(M_0+tP(\sup_{[0,t]}E)).
  \label{ghjj}
\end{equation}
Now, since 
\begin{equation}
  \dfrac{1}{\sqrt{\rho_0}}[\rho_0^2\partial_t^3v']' = \rho_0^{\frac{3}{2}}\partial_t^3v''+2\sqrt{\rho_0}\rho_0'\partial_t^3v',
  \label{}
\end{equation}
the estimate (\ref{rtu}) and (\ref{ghjj}) also imply that
\begin{equation}
  \sup_{[0,t]}||\sqrt{\rho_0}\rho_0'\partial_t^3v'||_{L^2}^2 \leq M_0+CtP(\sup_{[0,t]}E)+C\sup_{[0,t]}E^{\alpha}(M_0+tP(\sup_{[0,t]}E)).
  \label{}
\end{equation}
Therefore,
\begin{equation}
\begin{split}
 &\ \ \ \ \sup_{[0,t]}[||\rho_0^{\frac{3}{2}}\partial_t^3v''||_{L^2}^2+||\sqrt{\rho_0}\rho_0'\partial_t^3v'||_{L^2}^2+||\rho_0^{\frac{3}{2}}\partial_t^3v'||_{L^2}^2] \\&\leq M_0+CtP(\sup_{[0,t]}E))+C\sup_{[0,t]}E^{\alpha}(M_0+tP(\sup_{[0,t]}E))
  \label{}
\end{split}
\end{equation}
so that with (\ref{cond_2}) and (\ref{cond_3})
\begin{equation}
\begin{split}
&\ \ \ \  \sup_{[0,t]}[||\rho_0^{\frac{3}{2}}\partial_t^3v''||_{L^2}^2+||\sqrt{\rho_0}\partial_t^3v'||_{L^2}^2] \\&\leq M_0+CtP(\sup_{[0,t]}E))+C\sup_{[0,t]}E^{\alpha}(M_0+tP(\sup_{[0,t]}E))
\end{split}
  \label{}
\end{equation}
Together with (\ref{uu}) and the weighted embedding estimate (\ref{emb}), the above inequality shows that
\begin{equation}
\begin{split}
&\ \ \ \   \sup_{[0,t]}[||\rho_0^{\frac{3}{2}}\partial_t^3v''||_{L^2}^2+||\partial_t^3v||_{H^{\frac{1}{2}}}^2] \\ &\leq M_0+CtP(\sup_{[0,t]}E))+C\sup_{[0,t]}E^{\alpha}(M_0+tP(\sup_{[0,t]}E)).
\end{split}
  \label{last1}
\end{equation}

By studying the $\sqrt{\rho_0}\partial_x(\dfrac{1}{\rho_0}\partial_t^2)$-problem of (\ref{tytyt}) in the same manner, we find that
\begin{equation}
\begin{split}
&\ \ \ \   \sup_{[0,t]}[||\rho_0^{\frac{3}{2}}\partial_tv'''||_{L^2}^2+||\partial_tv||_{H^{\frac{3}{2}}}^2] \\&\leq M_0+CtP(\sup_{[0,t]}E))+C\sup_{[0,t]}E^{\alpha}(M_0+tP(\sup_{[0,t]}E)).
\end{split}
  \label{last2}
\end{equation}

\section{Proof of Theorem}\label{s7}
\subsection{Time of existence and bounds independent of $\kappa$ and existence of solutions to (\ref{simsystem})}
Summing the inequality (\ref{energyest}),(\ref{est_for_vtt}),(\ref{last0}),(\ref{last1}),(\ref{last2}), we find that
\begin{equation}
  \sup_{[0,t]}E(t) \leq M_0+CtP(\sup_{[0,t]}E))+C\sup_{[0,t]}E^{\alpha}(M_0+tP(\sup_{[0,t]}E)).
  \label{}
\end{equation}
As $\alpha<1$, by using Young's inequality and readjusting the constants, we obtain
\begin{equation}
  \sup_{[0,t]}E(t) \leq M_0+CtP(\sup_{[0,t]}E).
  \label{}
\end{equation}
This provides us with a time of existence $T_1$ independent of $\kappa$ and an estimate on $(0,T_1)$ independent of $\kappa$ of the type:
\begin{equation}
  \sup_{[0,T_1]}E(t) \leq 2M_0.
  \label{bound}
\end{equation}
In particular, our sequence of solutions $(v_{\kappa})$ satisfy the $\kappa$-independent bound (\ref{bound}) on the $\kappa$-independent time interval $(0,T_1)$.

\subsection{The limit as $\kappa \to 0$}
By the $\kappa$-independent estimate (\ref{bound}), there exists a subsequence of $(v_{\kappa})$ which converges weakly to $v$ in $L^2(0,T;H^2(I))$. With $\eta(x,t)=x+\int_0^tv(x,s)ds$, by standard compactness arguments, we see that a further subsequence of $v_{\kappa}$ and $\eta_{\kappa}'$ uniformly converges to $v$ and $\eta'$, respectively, which shows $v$ is the solution to (\ref{simsystem})--(\ref{90909}) and $v(x,0)=u_0(x)$.

\section{The general case for $1 <\gamma< 3$}
\label{ga3}
If $\gamma \neq 2$, we set $\omega_0=\rho_0^{\gamma-1}$, then physical vacuum condition shows that
\begin{equation}
  \omega_0 \geq C \text{dist}(x, \partial I),
\label{w0000}
\end{equation}
when $x \in I$ near the vacuum boundary $\Gamma$, and
\begin{align}
  \bigg |\dfrac{\partial\omega_0}{\partial x}(x)\bigg |\geq C\ \ \text{when}  \ d(x,\partial I) \leq \alpha,
  \label{cond_2}\\
  \omega_0 \geq C_{\alpha} > 0\ \ \text{when}  \ d(x,\partial I) \geq \alpha.
  \label{cond_3}
\end{align}

Now we can use $\omega_0 v$ as intermediate variable and construct approximate solution to degenerate parabolic regularization just in a similar way to Section \ref{s5}. Noticing that the force term $F=\int_0^x \rho_0 dy +C$ would not be smooth now, we need to require a certain high space regularity for it to keep the method described in Section \ref{s5} still work. Since we would require that $X=\omega_0 v \in H^3(I)$, then from (\ref{531}) and (\ref{initialdata}), we will need the regularity that $\omega_0\int_0^x\rho_0dy \in H^3(I)$. 

With (\ref{w0000}) and $\rho_0=\omega_0^{\frac{1}{\gamma-1}}$, we will just require that $\omega_0^{\frac{1}{\gamma-1}-1} \in L^2(I)$,  which means $1 < \gamma < 3$.

So for $1 < \gamma < 3$, we can get the local wellposedness by doing the similar proof as $\gamma=2$ in Section \ref{s5} -- \ref{s7}. Details can be seen in \cite{DS_2009}.
%
\section*{Acknowledgments}The work was in part supported by NSFC
(grants No. 10801029 and 10911120384), FANEDD, Shanghai Rising
Star Program (10QA1400300), SGST 09DZ2272900 and SRF for ROCS,
SEM.

\end{document}